\documentclass[11pt,a4paper]{amsart}
\usepackage[T1]{fontenc}
\usepackage{lmodern}
\usepackage{amsmath,amssymb,amsthm}
\usepackage{microtype}
\usepackage[margin=26mm]{geometry}
\usepackage{xcolor}
\usepackage[hidelinks]{hyperref}

\newif\ifshowchanges
\showchangestrue
\newcommand{\revisioncolor}{\ifshowchanges\color{red}\fi}

\newcommand{\doi}[1]{doi:\,\href{https://doi.org/#1}{\nolinkurl{#1}}}

\newtheorem{theorem}{Theorem}[section]
\newtheorem{lemma}[theorem]{Lemma}
\newtheorem{proposition}[theorem]{Proposition}
\newtheorem{corollary}[theorem]{Corollary}
\newtheorem{conjecture}[theorem]{Conjecture}
\theoremstyle{definition}
\newtheorem{definition}[theorem]{Definition}
\newtheorem{construction}[theorem]{Construction}
\theoremstyle{remark}

\numberwithin{equation}{section}

\title
{Counterexamples to Two Conjectures on Tangles
in Matroid Minors}

\author[Zhen Chen]{Zhen Chen}
\address{School of Mathematical Sciences, Xiamen University}
\email{chenzzhen@126.com}

\author[Zhuo Li]{Zhuo Li$^{*}$}
\address{School of Mathematical Sciences, Xiamen University}
\email{lzhuo@stu.xmu.edu.cn}
\thanks{$^{*}$Corresponding author.}

\subjclass[2020]{05B35}
\keywords{Matroid, tangle, breadth, truncation, weak $4$-connectivity, internal $4$-connectivity}
\date{}

\hypersetup{
  pdftitle={Counterexamples to conjectures on matroid tangle breadth},
  pdfauthor={Zhen Chen and Zhuo Li},
  pdfsubject={Counterexamples to simultaneous breadth preservation and internal 4-connectivity}
}

\begin{document}
\begin{abstract}
We give counterexamples to two conjectures of Brettell, Jowett,
Oxley, Semple, and Whittle concerning tangles in matroid minors.
We construct a matroid $P_0$ with a tangle $\mathcal T$
of order seven. For each $i\in\{4,5,6,7\}$, let
$\mathcal T_i$ denote its order-$i$ truncation.
No proper minor $N$ of $P_0$ admits, for every
such $i$, a tangle of order $i$ whose breadth equals that of
$\mathcal T_i$, and $P_0$ is not weakly $4$-connected.
We also construct a matroid $P_1$ with a tangle
of order four and breadth six, but no internally
$4$-connected minor of $P_1$ admits a tangle of order four.

\end{abstract}
\maketitle

\section{Introduction}
\label{sec1:introduction}

The study of matroid connectivity under deletion and contraction
goes back to Tutte~\cite{Tutte66}; for general background,
see Oxley~\cite{Oxley11}. Tangles, introduced for graphs by
Robertson and Seymour~\cite{RS91}, describe connectivity by consistent
choices across separations of low order. Tangles were later
extended to matroids by Dharmatilake~\cite{Dha96}.
Their connections with matroid branch-width and
tree-decompositions were further developed by Geelen, Gerards, Robertson and
Whittle~\cite{GGRW06}. Geelen, Gerards and Whittle~\cite{GGW09}
subsequently proved decomposition theorems that distinguish maximal
tangles and related high-order tangles to grid minors in matroids
representable over a finite field. 
Tangle matroids also feature in the
work of Geelen and van Zwam~\cite{GvZ15} on preserving $3$-connectivity
and a prescribed minor under deletion or contraction.
Hall~\cite{Hall15}
characterised the matroids that arise as tangle matroids.

A related question is whether a given tangle can be retained in a minor
with stronger connectivity properties. Carmesin and Kurkofka~\cite{CK25} characterised graph tangles
of order four through connectivity properties of minors.
For matroids, Brettell, Jowett, Oxley, Semple and
Whittle~\cite{BJO25} proved that every tangle of
order at least four generates a tangle of the same order
and breadth in a weakly $4$-connected minor.
Their theorem applies separately to each truncation of
order at least four, but does not guarantee a common minor.
They conjectured that all these truncations generate
tangles of the same respective orders and breadths in
a single weakly $4$-connected minor.

\begin{conjecture}[{\cite[Conjecture~84]{BJO25}}]
\label{sec1:conjecture}
Let $\mathcal T$ be a tangle of order $k\ge4$ in a finite matroid $M$.
For $4\le i\le k$, let
\[
 \mathcal T_i=\{A\in\mathcal T:\lambda_M(A)\le i-2\}.
\]
There is a weakly $4$-connected minor $N$ of $M$ such that each
$\mathcal T_i$ generates an order-$i$ tangle $\mathcal S_i$ in $N$ and $\operatorname{breadth}_N(\mathcal S_i)
 =\operatorname{breadth}_M(\mathcal T_i).$
\end{conjecture}

We disprove this conjecture by constructing a matroid
for which no proper minor admits tangles of all the
required orders and breadths simultaneously.

\begin{theorem}\label{sec1:main}
The matroid $P_0$ in Construction~\ref{sec3:construction}
is simple of rank six but is not weakly $4$-connected.
The collection
\[
\mathcal T=\{A\subseteq E(P_0):r_{P_0}(A)<6\}
\]
is an order-seven tangle with $(P_0)_{\mathcal T}=P_0$.
For $4\le i\le7$, let $\mathcal T_i$ denote its
order-$i$ truncation.
If a minor $N$ of $P_0$ admits an order-$i$ tangle
$\mathcal S_i$ for each $i\in\{4,5,6,7\}$ such that $\operatorname{breadth}_N(\mathcal S_i)
=\operatorname{breadth}_{P_0}(\mathcal T_i),$
then $N=P_0$. Consequently,
Conjecture~\ref{sec1:conjecture} is false.
\end{theorem}

The matroid $P_0$ is obtained by truncating a rational
vector matroid to rank six. We show that deleting any
element decreases the maximum size of an $s$-uniform
witness for some $s\in\{3,4,5,6\}$, while the rank
requirement for an order-seven tangle rules out nonempty
contractions in any minor preserving the prescribed breadths.

Brettell et al.~\cite{BJO25} also conjectured that
an internally $4$-connected minor could be obtained
while retaining at least half the breadth of the
original tangle.

\begin{conjecture}[{\cite[Conjecture~85]{BJO25}}]
\label{sec1:conjecture-internal}
Let $\mathcal T$ be a tangle of order $k\ge4$ and breadth $m$
in a matroid $M$. Then $M$ has an internally $4$-connected
minor $N$ in which $\mathcal T$ generates a tangle $\mathcal T'$
of order $k$ satisfying
$\operatorname{breadth}_N(\mathcal T')\ge m/2$.
\end{conjecture}

The conjecture fails already for $k=4$.

\begin{theorem}\label{sec1:main-internal}
The matroid $P_1$ in Construction~\ref{sec5:construction}
has a tangle of order four and breadth six, but no
internally $4$-connected minor of $P_1$ admits a tangle
of order four. Consequently,
Conjecture~\ref{sec1:conjecture-internal} is false.
\end{theorem}

Section~\ref{sec2:preliminaries} gives the definitions and tangle
facts used in both arguments. Section~\ref{sec3:construction-section}
constructs $P_0$ and determines its maximum witness sizes.
Section~\ref{sec4:counterexample} proves Theorem~\ref{sec1:main}.
Section~\ref{sec5:internal} constructs $P_1$ and proves
Theorem~\ref{sec1:main-internal}.

\section{Preliminaries}
\label{sec2:preliminaries}

All matroids are finite. For a matroid $P$, write $E(P)$ for its \emph{ground set}, $r_P$ for its
\emph{rank function}, and $r(P)=r_P(E(P))$ for its rank. A set $X\subseteq E(P)$
is \emph{independent} if $r_P(X)=|X|$, and it \emph{spans} $P$ if $r_P(X)=r(P)$.
A set that does not span $P$ is \emph{nonspanning}. A matroid $P$ is \emph{simple} if every subset of $E(P)$
with at most two elements is independent.
A set $X\subseteq E(P)$ is a \emph{flat} if
$r_P(X\cup\{e\})>r_P(X)$ for every $e\in E(P)\setminus X$. A rank-two flat is called a \emph{line}.

A matroid is \emph{representable over a field $\mathbb F$} if
its elements can be assigned vectors over $\mathbb F$ so that a
subset is independent exactly when the corresponding vectors are
linearly independent. The resulting matroid is called a \emph{vector matroid}.

The \emph{connectivity function} of $P$ is
\[
  \lambda_P(X)=r_P(X)+r_P(E(P)\setminus X)-r(P)
  \qquad (X\subseteq E(P)).
\]
In particular,
\[
  \lambda_P(X)=\lambda_P(E(P)\setminus X)
  \quad\text{and}\quad
  \lambda_P(X)\le r_P(X).
\]
We write $P|X$, $P\setminus D$, and $P/K$ for \emph{restriction, deletion}, and
\emph{contraction}, respectively. A matroid $N$ is a \emph{minor} of $P$
if $N=(P/K)\setminus D$ for disjoint subsets $K,D\subseteq E(P)$.
For $1\le s\le r(P)$, the \emph{rank-$s$ truncation} $\operatorname{Tr}_s(P)$ has ground set $E(P)$ and rank function
\[
  r_{\operatorname{Tr}_s(P)}(X)=\min\{s,r_P(X)\}.
\]
The \emph{uniform matroid} of rank $s$ on $m$ elements, where $0\le s\le m$, is
denoted by $U_{s,m}$; its independent sets are precisely the subsets of
size at most $s$.

\begin{definition}\label{sec2:def-witness}
Let $s$ be a positive integer. An \emph{$s$-uniform witness} in $P$ is a
set $W\subseteq E(P)$ with $|W|\ge s$ such that every subset of $W$ of size
at most $s$ is independent in $P$. Let $\beta_s(P)$ be the maximum size of
an $s$-uniform witness in $P$, with $\beta_s(P)=0$ if no such witness
exists.
\end{definition}

Let $1\le s\le r(P)$ and let $W\subseteq E(P)$ with $|W|\ge s$.
Then $W$ is an $s$-uniform witness in $P$ if and only if
\[
  \operatorname{Tr}_s(P)|W=U_{s,|W|}.
\]

\begin{definition}[\cite{RS91}]\label{sec2:def-tangle}
A \emph{tangle of order $k\ge2$} in $P$ is a collection
$\mathcal T\subseteq 2^{E(P)}$ satisfying the following axioms.
\begin{enumerate}
  \renewcommand{\labelenumi}{\textup{(T\arabic{enumi})}}
  \item $\lambda_P(A)<k-1$ for every $A\in\mathcal T$.
  \item For every $A\subseteq E(P)$ with $\lambda_P(A)<k-1$, either
    $A\in\mathcal T$ or $E(P)\setminus A\in\mathcal T$.
  \item For every $A,B,C\in\mathcal T$, not necessarily distinct,
    $A\cup B\cup C\ne E(P)$.
  \item $E(P)\setminus\{e\}\notin\mathcal T$ for every $e\in E(P)$.
\end{enumerate}
A member of $\mathcal T$ is called \emph{$\mathcal T$-small}.
\end{definition}

Axiom~(T3) implies that $E(P)\notin\mathcal T$ and that $\mathcal T$
cannot contain both a set and its complement. Consequently, whenever
$\lambda_P(A)<k-1$, exactly one of $A$ and $E(P)\setminus A$ belongs to
$\mathcal T$.

The \emph{tangle matroid} $P_{\mathcal T}$ has ground set $E(P)$ and rank
function
\begin{equation}\label{sec2:eq-tangle-rank}
  r_{P_{\mathcal T}}(X)
  =\min\Bigl(
    \{k-1\}\cup
    \{\lambda_P(A):X\subseteq A,\ A\in\mathcal T\}
  \Bigr);
\end{equation}
see \cite{BJO25}. In particular, since
$E(P)\notin\mathcal T$, no member of $\mathcal T$ contains $E(P)$, and hence
\[
  r(P_{\mathcal T})
  =r_{P_{\mathcal T}}(E(P))
  =\min\bigl(\{k-1\}\cup\varnothing\bigr)
  =k-1.
\]

For a matroid $Q$ of rank $r\ge1$, a restriction $Q|W$ is a
\emph{spanning uniform restriction} if $Q|W=U_{r,|W|}.$
The \emph{breadth} of $\mathcal T$ is the maximum size of a spanning
uniform restriction of $P_{\mathcal T}$. Thus
\[
  \operatorname{breadth}_P(\mathcal T)
  =\beta_{k-1}(P_{\mathcal T}).
\]

For $2\le i\le k$, the \emph{order-$i$ truncation} of $\mathcal T$ is
\[
  \mathcal T_i
  =\{A\in\mathcal T:\lambda_P(A)\le i-2\}.
\]
It is a tangle of order $i$, and by \cite[Lemma~83]{BJO25},
\begin{equation}\label{sec2:eq-tangle-truncation}
  P_{\mathcal T_i}
  =\operatorname{Tr}_{i-1}(P_{\mathcal T}).
\end{equation}
At $i=k$, we have
$\mathcal T_k=\mathcal T$ and
$\operatorname{Tr}_{k-1}(P_{\mathcal T})=P_{\mathcal T}$.

Let $N$ be a minor of $P$. We say that $\mathcal T$
\emph{generates} a tangle $\mathcal S$ of order $k$ in $N$
if $\mathcal S$ is the unique tangle of order $k$ in $N$ satisfying
\[
  \{A\cap E(N):A\in\mathcal T\}\subseteq\mathcal S.
\]

A \emph{$k$-separation} of $P$ is a partition
$E(P)=A\sqcup B$ with $|A|,|B|\ge k$ and $\lambda_P(A)\le k-1$.
A matroid is \emph{$k$-connected} if it has no $j$-separation
for any $j\in\{1,\ldots,k-1\}$.

A matroid $P$ is \emph{internally $4$-connected}
if it is $3$-connected and $\lambda_P(X)\ge3$
whenever $|X|,|E(P)\setminus X|\ge4$.
It is \emph{weakly $4$-connected} if it is
$3$-connected and $\lambda_P(X)\ge3$
whenever $|X|,|E(P)\setminus X|\ge5$.

The following lemma provides a rank inequality used
to bound the breadth of tangles in minors. This inequality
also follows from the quotient property of tangle matroids
\cite[Lemma~21]{BJO25}; we give a direct proof.

\begin{lemma}\label{sec2:lem-rank-comparison}
Let $\mathcal S$ be a tangle of order $k\ge3$ in a matroid $P$. Every set
$A\subseteq E(P)$ with $r_P(A)\le k-2$ belongs to $\mathcal S$. Moreover, for $X\subseteq E(P)$, 
\begin{equation}\label{sec2:eq-rank-comparison}
  r_{P_{\mathcal S}}(X)
  \le\min\{k-1,r_P(X)\}.
\end{equation}
In particular, $r(P)\ge k-1$.
\end{lemma}

\begin{proof}
Axiom~(T3) gives $E(P)\notin\mathcal S$. Since
$\lambda_P(\varnothing)=0<k-1$, axiom~(T2) then gives
$\varnothing\in\mathcal S$. For every $e\in E(P)$,
\[
  \lambda_P(\{e\})\le r_P(\{e\})\le1<k-1.
\]
Axioms~(T2) and~(T4) therefore imply that $\{e\}\in\mathcal S$.

We prove the first assertion by induction on $|A|$. The cases
$|A|\le1$ have just been established. Suppose that $|A|\ge2$ and
$r_P(A)\le k-2$, and choose $e\in A$. Since
$r_P(A\setminus\{e\})\le r_P(A)\le k-2$, the induction hypothesis gives
$A\setminus\{e\}\in\mathcal S$. If $A\notin\mathcal S$, then
\[
  \lambda_P(A)\le r_P(A)\le k-2<k-1,
\]
and axiom~(T2) gives $E(P)\setminus A\in\mathcal S$. The three members
$A\setminus\{e\}$, $\{e\}$, and $E(P)\setminus A$ of $\mathcal S$ would
then cover $E(P)$, contrary to~(T3). Hence $A\in\mathcal S$.

If $r_P(X)\le k-2$, the first assertion gives $X\in\mathcal S$, so by \eqref{sec2:eq-tangle-rank},
\[
 r_P(X)\ge \lambda_P(X) \ge  r_{P_{\mathcal S}}(X) .
\]
If $r_P(X)\ge k-1$, then $r_P(X)\ge k-1 \ge r_{P_{\mathcal S}}(X)$. This proves
\eqref{sec2:eq-rank-comparison}. Taking $X=E(P)$ gives
\[
 r(P)=r_P(E(P))\ge r_{P_{\mathcal S}}(E(P))=  k-1 .
\]
\end{proof}

\begin{corollary}\label{sec2:cor-minor-witnesses}
Let $N$ be a minor of $P$, and let $\mathcal S$ be a tangle of order
$s+1\ge3$ in $N$. If $N_{\mathcal S}|W$ is a spanning uniform restriction
of $N_{\mathcal S}$, then $W$ is an $s$-uniform witness in $P|E(N)$.
Consequently,
\begin{equation}\label{sec2:eq-minor-breadth}
  \operatorname{breadth}_N(\mathcal S)
  \le\beta_s(P|E(N)).
\end{equation}
\end{corollary}

\begin{proof}
Write $N=P/K\setminus D$, where $K\cap D=\varnothing$. For
$X\subseteq E(N)$, the contraction rank formula gives
\[
  r_N(X)=r_P(X\cup K)-r_P(K).
\]
Submodularity of the rank function gives
\[
 r_P(X\cup K)+r_P(X\cap K)\le r_P(X)+r_P(K).
\]
Since $X\cap K=\varnothing$, this implies
$r_P(X\cup K)\le r_P(X)+r_P(K)$ and hence $r_N(X)\le r_P(X)$.

Suppose that $N_{\mathcal S}|W=U_{s,|W|}$, where $|W|\ge s$.
For every $X\subseteq W$ with $|X|\le s$,
Lemma~\ref{sec2:lem-rank-comparison} gives
\[
  |X|=r_{N_{\mathcal S}}(X)
  \le r_N(X)\le r_P(X)\le |X|.
\]
Thus $r_P(X)=|X|$, so $X$ is independent in $P$. Since restriction
preserves independence, $W$ is an $s$-uniform witness in $P|E(N)$.
Maximizing $|W|$ proves \eqref{sec2:eq-minor-breadth}.
\end{proof}

\begin{lemma}\label{sec2:lem-large-witness}
Let $P$ be a matroid of rank $r\ge3$ with an $r$-uniform
witness $W$ satisfying $|W|>3(r-1)$. Then
\[
\mathcal T=\{A\subseteq E(P):r_P(A)<r\}
\]
is a tangle of order $r+1$ in $P$ and $P_{\mathcal T}=P$.
Moreover, for $4\le i\le r+1$, $\operatorname{breadth}_P(\mathcal T_i)=\beta_{i-1}(P).$
\end{lemma}

\begin{proof}
Every nonspanning set meets $W$ in at most $r-1$ elements,
since any $r$ elements of $W$ form an independent set.
Thus, for $A,B,C\in\mathcal T$,
\[
|(A\cup B\cup C)\cap W|
\le 3(r-1)<|W|.
\]
Hence $A\cup B\cup C\ne E(P)$, proving~(T3).

For each $A\in\mathcal T$, the complement $E(P)\setminus A$
spans $P$. Otherwise, $A$, $A$, and $E(P)\setminus A$
would violate~(T3). Consequently,
\begin{equation}\label{sec2:eq-nonspanning-connectivity}
\lambda_P(A)=r_P(A)<r
\qquad (A\in\mathcal T),
\end{equation}
which proves~(T1).

To verify~(T2), suppose that $\lambda_P(A)<r$.
If $A\notin\mathcal T$, then $r_P(A)=r$, and hence
\[
r_P(E(P)\setminus A)=\lambda_P(A)<r.
\]
Thus $E(P)\setminus A\in\mathcal T$, as required.

For each $e\in E(P)$, the set $W\setminus\{e\}$ has
at least $r$ elements and therefore contains an independent
$r$-element set. It follows that $E(P)\setminus\{e\}$
spans $P$ and does not belong to $\mathcal T$.
This proves~(T4), so $\mathcal T$ is a tangle of order $r+1$.

We next prove $P_{\mathcal T}=P$ by comparing ranks.
Let $X\subseteq E(P)$ be nonspanning. Then $X\in\mathcal T$,
and every $A\in\mathcal T$ containing $X$ satisfies
\[
\lambda_P(A)=r_P(A)\ge r_P(X),
\]
with equality when $A=X$.
Since $r_P(X)<r$, the rank formula
\eqref{sec2:eq-tangle-rank} gives
$r_{P_{\mathcal T}}(X)=r_P(X)$.
If $X$ spans $P$, no member of $\mathcal T$ contains $X$,
so the same formula gives
$r_{P_{\mathcal T}}(X)=r=r_P(X)$.
Thus the two rank functions agree on every subset of $E(P)$,
and $P_{\mathcal T}=P$.

For $4\le i\le r+1$, equation~\eqref{sec2:eq-tangle-truncation}
now gives
\begin{equation}\label{sec2:eq-canonical-truncations}
P_{\mathcal T_i}
=\operatorname{Tr}_{i-1}(P_{\mathcal T})
=\operatorname{Tr}_{i-1}(P).
\end{equation}
Rank-$(i-1)$ truncation preserves the independence of every
set of size at most $i-1$. Hence $P$ and
$\operatorname{Tr}_{i-1}(P)$ have the same
$(i-1)$-uniform witnesses. By the definition of breadth,
\begin{align*}
\operatorname{breadth}_P(\mathcal T_i)
=\beta_{i-1}(P_{\mathcal T_i})
=\beta_{i-1}(\operatorname{Tr}_{i-1}(P))
=\beta_{i-1}(P).
\end{align*}
\end{proof}

\section{The construction of $P_0$}
\label{sec3:construction-section}
All vector spaces in this section are over $\mathbb Q$.

\begin{construction}\label{sec3:construction}
Set $(t_1,t_2,t_3)=(3,4,5)$. Let $V$ be a vector space
over $\mathbb Q$ with basis
\[
\begin{split}
\mathcal B={}&
\{u,v,\xi_1,\xi_2,\xi_3\}\\
&\cup
\{a_{j,h,\ell}:
  1\le j\le3,\ 1\le h\le2,\ 1\le\ell\le t_j-1\}\\
&\cup
\{d_{a,h,\ell}:
  a\in\{4,5\},\ 1\le h\le2,\ 1\le\ell\le2\}.
\end{split}
\]
All the listed basis vectors are distinct.

Define
\[
\begin{gathered}
f_1=u,\qquad f_2=v,\qquad f_3=u+v,\qquad f_4=u+2v,\qquad f_5=u+3v,
\end{gathered}
\]
and set
\[
H_0=\operatorname{span}\{u,v\},
\qquad
F=\{f_1,\ldots,f_5\}.
\]
For $j\in\{1,2,3\}$, define
\[
\eta_j=\xi_j+f_j.
\]

For each $j\in\{1,2,3\}$ and $h\in\{1,2\}$, define
\[
B_{j,h}
=
\left\{
a_{j,h,1},\ldots,a_{j,h,t_j-1},
\ \xi_j-\sum_{\ell=1}^{t_j-1}a_{j,h,\ell}
\right\}.
\]
For each $a\in\{4,5\}$ and $h\in\{1,2\}$, define
\[
D_{a,h}
=
\{d_{a,h,1},d_{a,h,2},
  f_a-d_{a,h,1}-d_{a,h,2}\}.
\]
We call each specified set $B_{j,h}$ or $D_{a,h}$
a \emph{petal}. Each $B_{j,h}$ has $t_j$ elements,
and each $D_{a,h}$ has three.

The vectors listed in $F$, in the sets
$\{\xi_j,\eta_j\}$, and in the petals are all distinct.
We form the following disjoint unions,
denoted by $\sqcup$:
\[
\begin{aligned}
A_j
&=\{\xi_j,\eta_j\}\sqcup B_{j,1}\sqcup B_{j,2}
&& (j=1,2,3),\\
D_a
&=D_{a,1}\sqcup D_{a,2}
&& (a=4,5),\\
G
&=A_1\sqcup A_2\sqcup A_3\sqcup D_4\sqcup D_5,\\
E_0
&=F\sqcup G.
\end{aligned}
\]

Let $P_{\mathrm{aux}}$ be the vector matroid with
ground set $E_0$, whose independent sets are precisely
the subsets of $E_0$ that are linearly independent in $V$.

The vectors
$f_1,f_2,\xi_1,\xi_2,\xi_3,a_{1,1,1}$
are linearly independent, so $r(P_{\mathrm{aux}})\ge6$.
Define $P_0=\operatorname{Tr}_6(P_{\mathrm{aux}}).$
Then $P_0$ has ground set $E_0$ and rank function
\[
r_{P_0}(X)
=
\min\{6,r_{P_{\mathrm{aux}}}(X)\}
\qquad (X\subseteq E_0).
\]
\end{construction}

\begin{proposition}\label{sec3:prop-basic}
The matroid $P_0$ is simple of rank six, and $F$ is a
rank-two flat of $P_0$. Every set $X\subseteq E_0$ with
$|X|\le6$ is independent in $P_0$ if and only if it is
independent in $P_{\mathrm{aux}}$.
\end{proposition}

\begin{proof}
The construction gives $r(P_{\mathrm{aux}})\ge6$, so $r(P_0)=6$.
If $X\subseteq E_0$ and $|X|\le6$, then
$r_{P_{\mathrm{aux}}}(X)\le |X|\le6$ and hence
\[
r_{P_0}(X)=\min\{6,r_{P_{\mathrm{aux}}}(X)\}
         =r_{P_{\mathrm{aux}}}(X).
\]
Thus a set of at most six elements is
independent in $P_0$ if and only if it is independent in
$P_{\mathrm{aux}}$.

To prove simplicity, it can be checked that every vector in $E_0$
is nonzero and every subset of $E_0$ of size at most two is linearly
independent. This shows that $P_{\mathrm{aux}}$ is simple, and the
preservation of independence up to size six proves that $P_0$ is simple.

The set $F$ spans $H_0$ and has rank two in $P_0$.
Adjoining any vector of $E_0\setminus F$ to $F$ increases its rank
from two to three. Thus $F$ is a flat of $P_0$.
\end{proof}

\begin{lemma}\label{sec3:lem-relations}
Let $c_x\in\mathbb Q$ for $x\in E_0$.
Then $\sum_{x\in E_0}c_xx=0$ if and only if the coefficients
are constant on each petal and satisfy
\begin{equation}\label{sec3:eq-relations}
\begin{aligned}
c_{\xi_j}+c_{\eta_j}+\alpha_{j,1}+\alpha_{j,2}&=0
    &&(j=1,2,3),\\
\sum_{j=1}^{3}c_{\eta_j} f_j
+\sum_{a=4}^{5}(\delta_{a,1}+\delta_{a,2})f_a
+\sum_{a=1}^{5}c_{f_a} f_a&=0,
\end{aligned}
\end{equation}
where $\alpha_{j,h}$ and $\delta_{a,h}$ are the common
coefficients on $B_{j,h}$ and $D_{a,h}$, respectively.
\end{lemma}

\begin{proof}
Suppose first that $\sum_{x\in E_0}c_xx=0$.
Fix a petal $B_{j,h}$, and let $q$ be the coefficient of
$\xi_j-\sum_{\ell=1}^{t_j-1}a_{j,h,\ell}$.
Expanding this vector and rearranging the relation gives
\[
\sum_{\ell=1}^{t_j-1}
(c_{a_{j,h,\ell}}-q)a_{j,h,\ell}
=
-q\xi_j-\sum_{x\in E_0\setminus B_{j,h}}c_xx.
\]
Every vector on the right is a linear combination of the
basis vectors in $\mathcal B\setminus
\{a_{j,h,1},\ldots,a_{j,h,t_j-1}\}.$
The left side is a linear combination of the omitted basis
vectors. Since $\mathcal B$ is linearly independent, the
displayed equality implies
$c_{a_{j,h,\ell}}-q=0$ for every $\ell$.
Thus the coefficients are constant on $B_{j,h}$.

Similarly, fix $D_{a,h}$ and let $q$ be the coefficient of
$f_a-d_{a,h,1}-d_{a,h,2}$. Expanding and rearranging gives
\[
(c_{d_{a,h,1}}-q)d_{a,h,1}
+(c_{d_{a,h,2}}-q)d_{a,h,2}
=
-qf_a-\sum_{x\in E_0\setminus D_{a,h}}c_xx.
\]
Both $f_a$ and every vector in $E_0\setminus D_{a,h}$
are linear combinations of the basis vectors in $\mathcal B\setminus\{d_{a,h,1},d_{a,h,2}\}.$
Linear independence of $\mathcal B$ therefore gives
$c_{d_{a,h,1}}=c_{d_{a,h,2}}=q$.
Hence the coefficients are constant on every petal.

Write $\alpha_{j,h}$ and $\delta_{a,h}$ for the common
coefficients on $B_{j,h}$ and $D_{a,h}$, respectively.
Since
\[
\sum_{x\in B_{j,h}}x=\xi_j,
\qquad
\sum_{x\in D_{a,h}}x=f_a,
\]
substituting $\eta_j=\xi_j+f_j$ yields
\[
\begin{aligned}
\sum_{x\in E_0}c_xx
&=
\sum_{j=1}^{3}
(c_{\xi_j}+c_{\eta_j}+\alpha_{j,1}+\alpha_{j,2})\xi_j\\
&\quad+
\sum_{j=1}^{3}c_{\eta_j}f_j
+\sum_{a=4}^{5}(\delta_{a,1}+\delta_{a,2})f_a
+\sum_{a=1}^{5}c_{f_a}f_a.
\end{aligned}
\]
All terms involving the vectors $f_a$ lie in
$H_0=\operatorname{span}\{u,v\}$.
Since $u,v,\xi_1,\xi_2,\xi_3$ are linearly independent,
this expression vanishes only if each coefficient of
$\xi_j$ is zero and the sum of the terms in $H_0$ is zero.
These are precisely the equations
in~\eqref{sec3:eq-relations}, proving necessity.

The converse follows from the same expansion.
\end{proof}

For a subset $Y$ of $E_0$, a petal is \emph{full in $Y$} if it is contained in $Y$, and
\emph{incomplete in $Y$} otherwise.
For $Y\subseteq E_0$, write $\sigma(Y)=\sum_{x\in Y}x$.
Thus
\[
 \sigma(B_{j,h})=\xi_j,
 \qquad \sigma(D_{a,h})=f_a.
\]
Consequently, each set in the following list is linearly dependent in
$V$. Its size and an explicit linear relation are also shown:
\begin{equation}\label{sec3:eq-dependent-sets}
\renewcommand{\arraystretch}{1.2}
\begin{array}{ccl}
 \text{Set}&\text{Size}&\text{Relation}\\ \hline
 \{\xi_j\}\cup B_{j,h}&t_j+1&
       \sigma(B_{j,h})-\xi_j=0\\
 B_{j,1}\cup B_{j,2}&2t_j&
       \sigma(B_{j,1})-\sigma(B_{j,2})=0\\
 \{f_j,\xi_j,\eta_j\}&3&
       \eta_j-\xi_j-f_j=0\\
 \{f_j,\eta_j\}\cup B_{j,h}&t_j+2&
       \eta_j-f_j-\sigma(B_{j,h})=0\\
 \{f_a\}\cup D_{a,h}&4&
       \sigma(D_{a,h})-f_a=0\\
 D_{a,1}\cup D_{a,2}&6&
       \sigma(D_{a,1})-\sigma(D_{a,2})=0
\end{array}
\end{equation}
Here $j\in\{1,2,3\}$, $a\in\{4,5\}$, and $h\in\{1,2\}$. These sets
are therefore dependent in $P_{\mathrm{aux}}$, and remain dependent
in $P_0$ because truncation cannot increase their ranks.

\subsection{Upper bounds and explicit witnesses}

Recall that an $s$-uniform witness $W$ has $|W|\ge s$ and every subset
of $W$ of size at most $s$ is independent.

\begin{lemma}\label{sec3:lem-bounds}
Let $W$ be an $s$-uniform witness in $P_0$, where
$s\in\{3,4,5,6\}$. Then $|W\cap F|\le2$, and the sizes of
$W\cap A_1,W\cap A_2,W\cap A_3,W\cap D_4,W\cap D_5$ are bounded
by the corresponding entries of Table~\ref{sec3:tab-bounds}.
In particular, $|W\cap G|\le c_s$, where $c_s$ is the sum of the five
bounds in its row.
\end{lemma}
\begin{table}[htbp]
\centering
\caption{Upper bounds on the five intersections of an $s$-uniform witness.}
\label{sec3:tab-bounds}
\renewcommand{\arraystretch}{1.15}
\begin{tabular}{c|rrrrr|r}
 $s$&$|W\cap A_1|$&$|W\cap A_2|$&$|W\cap A_3|$&$|W\cap D_4|$&$|W\cap D_5|$&$c_s$\\ \hline
 3&$\le 8$&$\le 10$&$\le 12$&$\le 6$&$\le 6$&42\\
 4&$\le 7$&$\le10$&$\le12$&$\le6$&$\le6$&41\\
 5&$\le7$&$\le9$&$\le12$&$\le6$&$\le6$&40\\
 6&$\le6$&$\le9$&$\le11$&$\le5$&$\le5$ &36
\end{tabular}
\end{table}
\begin{proof}
Since $r_{P_0}(F)=2$, every three elements of $F$ are dependent.
Thus $|W\cap F|\le2$ for every $s\in\{3,4,5,6\}$. 

Recall that $(t_1,t_2,t_3)=(3,4,5)$. For $s=3$, the bounds follow from
$|W\cap A_j|\le |A_j|=2+2t_j$ for $j=1,2,3$ and
$|W\cap D_a|\le |D_a|=6$ for $a=4,5$. Suppose that $t_j+1\le s$. If $\xi_j\in W$, neither $B_{j,1}$ nor
$B_{j,2}$ can be full in $W$: otherwise $W$ would contain the dependent
set $\{\xi_j\}\cup B_{j,h}$ of size at most $s$. Each petal then
contains at most $t_j-1$ elements of $W$, and at most two elements
of $\{\xi_j,\eta_j\}$ are in $W$. If $\xi_j\notin W$, at most one element
of $\{\xi_j,\eta_j\}$ remains, while each petal contains at most
$t_j$ elements of $W$. Hence
\[
 |W\cap A_j|\le
 \begin{cases}
  2+2(t_j-1)=2t_j,&\xi_j\in W,\\
  1+2t_j,&\xi_j\notin W.
 \end{cases}
\]
For $j=1$, this gives $|W\cap A_1|\le 7$ whenever $s\ge4$; for $j=2$, it gives
$|W\cap A_2|\le 9$ whenever $s\ge5$; for $j=3$, it gives $|W\cap A_3|\le 11$ when $s=6$.

When $s=6$, the upper bound on $|W\cap A_1|$ improves to six. If $\xi_1\in W$,
the preceding inequality already gives six. If $\xi_1\notin W$,
the dependent six-element set $B_{1,1}\cup B_{1,2}$ cannot be
contained in $W$. At most five of its elements can therefore be
retained, and adding $\eta_1$ gives at most six. Likewise, the
dependent six-element set $D_{a,1}\cup D_{a,2}$ gives
$|W\cap D_a|\le5$ for $a=4,5$.
The remaining entries follow from the sizes of the sets. Since $G=A_1\sqcup A_2\sqcup A_3\sqcup D_4\sqcup D_5$, adding their bounds gives $c_s$.
\end{proof}

\begin{proposition}\label{sec3:prop-witnesses}
Define
\begin{equation}\label{sec3:eq-witnesses}
\begin{aligned}
 S_3&=G\cup\{f_4,f_5\},\\
 S_4&=(G\setminus\{\xi_1\})\cup\{f_1\},\\
 S_5&=(G\setminus\{\xi_1,\xi_2\})\cup\{f_2\},\\
 S_6&=\bigl(G\setminus
       \{\xi_1,\xi_2,\xi_3,a_{1,2,1},d_{4,2,1},d_{5,2,1}\}\bigr)
       \cup\{f_3\}.
\end{aligned}
\end{equation}
For each $s\in\{3,4,5,6\}$, the set $S_s$ is an $s$-uniform witness
in $P_0$, and
\[
 (|S_3|,|S_4|,|S_5|,|S_6|)=(44,42,41,37).
\]
\end{proposition}

\begin{proof}
It can be checked that
$(|S_3|,|S_4|,|S_5|,|S_6|)=(44,42,41,37)$.
Fix $s\in\{3,4,5,6\}$. By
Proposition~\ref{sec3:prop-basic}, it remains to show that
every subset of $S_s$ of size at most $s$ is linearly
independent in $V$.

Suppose otherwise. Choose a linearly dependent set
$X\subseteq S_s$ of minimum cardinality among those with
$|X|\le s$. Then every proper subset of $X$ is linearly
independent. Choose coefficients $c_x\in\mathbb Q$, not all
zero, such that $\sum_{x\in X}c_xx=0$.
Every $c_x$ is nonzero; otherwise, the vectors with nonzero
coefficients would form a linearly dependent proper subset
of $X$. Extend the coefficients to $E_0$ by setting
$c_x=0$ for $x\notin X$. Thus
$\sum_{x\in E_0}c_xx=\sum_{x\in X}c_xx=0$.

Define
\[
C_a=
\begin{cases}
S_s\cap(A_a\cup\{f_a\}),&a=1,2,3,\\
S_s\cap(D_a\cup\{f_a\}),&a=4,5.
\end{cases}
\]
These five sets partition $S_s$. For $a=1,\ldots,5$, let
$L_a=\sum_{x\in X\cap C_a}c_xx$.
The sets $X\cap C_a$ are pairwise disjoint and have union
$X$, so $\sum_{a=1}^{5}L_a=\sum_{x\in X}c_xx=0$.

By Lemma~\ref{sec3:lem-relations}, the coefficients are
constant on each petal. Let $\alpha_{a,h}$, for $a=1,2,3$,
and $\delta_{a,h}$, for $a=4,5$, denote the common
coefficients on $B_{a,h}$ and $D_{a,h}$, respectively.
If a common coefficient is nonzero, the entire petal
is contained in $X$, since coefficients outside $X$
are zero.

For $a=1,2,3$, $c_x=0$ for $x\notin X$. Hence
\[
\begin{aligned}
L_a
&=\sum_{x\in X\cap C_a}c_xx
 =\sum_{x\in A_a\cup\{f_a\}}c_xx\\
&=c_{\xi_a}\xi_a+c_{\eta_a}\eta_a
 +\alpha_{a,1}\sum_{x\in B_{a,1}}x
 +\alpha_{a,2}\sum_{x\in B_{a,2}}x+c_{f_a}f_a\\
&=(c_{\xi_a}+c_{\eta_a}
   +\alpha_{a,1}+\alpha_{a,2})\xi_a
  +(c_{\eta_a}+c_{f_a})f_a\\
&=(c_{\eta_a}+c_{f_a})f_a.
\end{aligned}
\]
The penultimate equality uses
$\sum_{x\in B_{a,h}}x=\xi_a$ and $\eta_a=\xi_a+f_a$;
the last follows from the first equation
in~\eqref{sec3:eq-relations}.

Similarly, for $a=4,5$, the definition of $C_a$ and
the zero extension give
\[
\begin{aligned}
L_a
&=\sum_{x\in X\cap C_a}c_xx
 =\sum_{x\in D_a\cup\{f_a\}}c_xx\\
&=\delta_{a,1}\sum_{x\in D_{a,1}}x
 +\delta_{a,2}\sum_{x\in D_{a,2}}x
 +c_{f_a}f_a\\
&=(\delta_{a,1}+\delta_{a,2}+c_{f_a})f_a,
\end{aligned}
\]
where $\sum_{x\in D_{a,h}}x=f_a$.
Consequently, $L_a=\rho_af_a$, where
$\rho_a=c_{\eta_a}+c_{f_a}$ for $a=1,2,3$ and
$\rho_a=\delta_{a,1}+\delta_{a,2}+c_{f_a}$ for $a=4,5$.
In particular, $\sum_{a=1}^{5}\rho_af_a=0$.

Let $I=\{a\in\{1,\ldots,5\}:X\cap C_a\ne\varnothing\}$.
Since $X$ is nonempty, $I$ is nonempty. We distinguish
two cases.

\smallskip
\noindent\textbf{Case 1.} $|I|=1$.

Let $a$ be the unique element of $I$. Then $X\subseteq C_a$
and $L_a=0$.

\smallskip
\noindent\emph{Subcase 1.1.} $a\in\{1,2,3\}$.

Since $L_a=0$ and $f_a\ne0$, we have
$c_{\eta_a}+c_{f_a}=0$.
The first equation in~\eqref{sec3:eq-relations} also gives
$c_{\xi_a}+c_{\eta_a}+\alpha_{a,1}+\alpha_{a,2}=0$.

Suppose first that $c_{\eta_a}=0$.
Then $c_{f_a}=0$ and
$c_{\xi_a}+\alpha_{a,1}+\alpha_{a,2}=0$.
If $c_{\xi_a}\ne0$, at least one $\alpha_{a,h}$ is
nonzero, so $\{\xi_a\}\cup B_{a,h}\subseteq X$.
If $c_{\xi_a}=0$, the two $\alpha$-coefficients must be
nonzero and opposite; otherwise all coefficients would
be zero. Hence $B_{a,1}\cup B_{a,2}\subseteq X$.

Now suppose that $c_{\eta_a}\ne0$.
Then $c_{f_a}=-c_{\eta_a}\ne0$, so $\eta_a,f_a\in X$.
If $c_{\xi_a}\ne0$, we obtain
$\{f_a,\xi_a,\eta_a\}\subseteq X$.
If $c_{\xi_a}=0$, then
$\alpha_{a,1}+\alpha_{a,2}=-c_{\eta_a}\ne0$.
Thus at least one $\alpha_{a,h}$ is nonzero, giving
$\{f_a,\eta_a\}\cup B_{a,h}\subseteq X$.

\smallskip
\noindent\emph{Subcase 1.2.} $a\in\{4,5\}$.

Since $L_a=0$ and $f_a\ne0$, we have
$\delta_{a,1}+\delta_{a,2}+c_{f_a}=0$.
If $c_{f_a}=0$, the two $\delta$-coefficients must be
nonzero and opposite, since the coefficients are not
all zero. Therefore $D_{a,1}\cup D_{a,2}\subseteq X$.
If $c_{f_a}\ne0$, at least one $\delta_{a,h}$ is nonzero,
so $\{f_a\}\cup D_{a,h}\subseteq X$.

\smallskip
Thus, in either subcase, $X$ contains one of the six
types of linearly dependent sets in~\eqref{sec3:eq-dependent-sets}.
We now check which of these sets can be contained in $S_s$.

For $S_3$, none of $f_1,f_2,f_3$ is present, so the
three-element sets $\{f_a,\xi_a,\eta_a\}$ cannot occur.
Every other type has at least four elements.

For $S_4$, only $f_1$ is selected from $F$, and $\xi_1$
is absent. Hence no three-element set of the listed
types is contained in $S_4$.
The four-element sets in the list are
$\{\xi_1\}\cup B_{1,h}$ and
$\{f_a\}\cup D_{a,h}$ for $a=4,5$.
These are also excluded, respectively by the absence
of $\xi_1$ and of $f_4,f_5$. Thus every listed set
contained in $S_4$ has at least five elements.

For $S_5$, only $f_2$ is selected from $F$, and
$\xi_1,\xi_2$ are absent. A set
$\{\xi_a\}\cup B_{a,h}$ can therefore occur only for
$a=3$, when its size is six. A set
$\{f_a,\eta_a\}\cup B_{a,h}$ can occur only for $a=2$,
again with size six. The unions of two $B$-petals or
two $D$-petals have at least six elements. The other
two types cannot occur. Hence every listed set contained
in $S_5$ has at least six elements.

For $S_6$, all $\xi_a$ are absent, only $f_3$ is selected
from $F$, and $B_{1,2},D_{4,2},D_{5,2}$ are incomplete.
The only listed sets that can be contained in $S_6$ are
$B_{2,1}\cup B_{2,2}$, $B_{3,1}\cup B_{3,2}$, and
$\{f_3,\eta_3\}\cup B_{3,h}$, with sizes eight, ten,
and seven, respectively.

In every case, the listed set contained in $X$ has
at least $s+1$ elements. Hence $|X|\ge s+1$, contradicting
$|X|\le s$. This excludes Case~1.

\smallskip
\noindent\textbf{Case 2.} $|I|\ge2$.

We first show that $L_a\ne0$ for every $a\in I$.
If $L_a=0$, then $\sum_{x\in X\cap C_a}c_xx=0$ is a
linear dependence on the nonempty proper subset
$X\cap C_a$ of $X$, since every coefficient on $X$
is nonzero. This contradicts the minimality of $X$.
Thus $\rho_a\ne0$ for every $a\in I$, and
$\sum_{a\in I}\rho_af_a=0$.
Any two distinct vectors in $F$ are linearly independent
by Proposition~\ref{sec3:prop-basic}. Therefore $|I|\ge3$.

For each $a\in I$, put $n_a=|X\cap C_a|$.
We obtain lower bounds for $n_a$ in the following
three subcases.

\smallskip
\noindent\emph{Subcase 2.1.} $f_a\in S_s$.

Since $a\in I$, we have $X\cap C_a\ne\varnothing$,
and hence $n_a=|X\cap C_a|\ge1$.

\smallskip
\noindent\emph{Subcase 2.2.}
$a\in\{1,2,3\}$ and $f_a\notin S_s$.

Since $X\subseteq S_s$, we have $f_a\notin X$. Thus $c_{f_a}=0$. Since $L_a=c_{\eta_a}f_a\ne0$,
we have $c_{\eta_a}\ne0$, so $\eta_a\in X$.
Moreover,
$c_{\xi_a}+\alpha_{a,1}+\alpha_{a,2}=-c_{\eta_a}\ne0$.

If $\xi_a\in S_s$, at least one of
$c_{\xi_a},\alpha_{a,1},\alpha_{a,2}$ is nonzero.
Thus $X$ contains at least one element of $A_a$
other than $\eta_a$, giving $n_a\ge2$.

If $\xi_a\notin S_s$, then $c_{\xi_a}=0$.
At least one $\alpha_{a,h}$ is therefore nonzero, so
$X$ contains $\eta_a$ and all $t_a$ elements of
$B_{a,h}$. Hence $n_a\ge t_a+1$.

\smallskip
\noindent\emph{Subcase 2.3.}
$a\in\{4,5\}$ and $f_a\notin S_s$.

Here $c_{f_a}=0$. Since
$L_a=(\delta_{a,1}+\delta_{a,2})f_a\ne0$,
at least one $\delta_{a,h}$ is nonzero.
The corresponding three-element petal is contained
in $X$, so $n_a\ge3$.

\smallskip
These subcases give the bounds in
Table~\ref{sec3:tab-nonzero}. An entry equal to one
corresponds to $f_a\in S_s$. For $a\le3$ with
$f_a\notin S_s$, the entry is two if $\xi_a\in S_s$,
and $t_a+1$ otherwise. For $a=4,5$ with $f_a\notin S_s$,
the entry is three.

\begin{table}[htbp]
\centering
\caption{Lower bounds on $n_a$ for $a\in I$.
The final column gives the resulting lower bound on
$|X|$, using $|I|\ge3$.}
\label{sec3:tab-nonzero}
\renewcommand{\arraystretch}{1.15}
\begin{tabular}{c|rrrrr|c}
$s$ & $n_1$ & $n_2$ & $n_3$ & $n_4$ & $n_5$
    & $|X|$\\ \hline
3 & $\ge$2 & $\ge$2 & $\ge$2 & $\ge$1 & $\ge$1 & $\ge$4\\
4 & $\ge$1 & $\ge$2 & $\ge$2 & $\ge$3 & $\ge$3 & $\ge5$\\
5 & $\ge$4 & $\ge$1 & $\ge$2 & $\ge$3 & $\ge$3 & $\ge6$\\
6 & $\ge$4 & $\ge$5 & $\ge$1 & $\ge$3 & $\ge$3 & $\ge7$
\end{tabular}
\end{table}

Recall $|X|=\sum_{a\in I}n_a$.
We have $|I|\ge3$, so this sum is at least the sum of
the three smallest entries in the relevant row of
Table~\ref{sec3:tab-nonzero}. These sums are
$4,5,6,7$ for $s=3,4,5,6$, respectively.
Thus $|X|\ge s+1$, again contradicting $|X|\le s$.

Both cases are impossible. Therefore, every subset of
$S_s$ of size at most $s$ is linearly independent in $V$.
By Proposition~\ref{sec3:prop-basic}, each such subset
is independent in $P_0$. Since $|S_s|\ge s$, the set
$S_s$ is an $s$-uniform witness, as required.
\end{proof}

\subsection{Maximum witnesses}

Recall that $\beta_s(P)$ is the maximum size of
an $s$-uniform witness in $P$. A \emph{maximum $s$-uniform witness} in $P$ is an $s$-uniform witness of size $\beta_s(P)$.

\begin{proposition}\label{sec3:prop-maxima}
The maximum witness sizes in $P_0$ are
\[
\bigl(\beta_3(P_0),\beta_4(P_0),
\beta_5(P_0),\beta_6(P_0)\bigr)
=(44,42,41,37).
\]
For each $s\in\{3,4,5,6\}$, every $s$-uniform witness
$W$ of size $\beta_s(P_0)$ satisfies the corresponding
condition in Table~\ref{sec3:tab-maxima}.
Moreover, $S_3$ is the unique maximum $3$-uniform witness.
\end{proposition}

\begin{table}[htbp]
\centering
\caption{The intersection $W\cap F$ for every maximum
$s$-uniform witness $W$.}
\label{sec3:tab-maxima}
\renewcommand{\arraystretch}{1.15}
\begin{tabular}{c|cccc}
$s$ & $3$ & $4$ & $5$ & $6$\\ \hline
$W\cap F$
& $\{f_4,f_5\}$
& $\{f_1\}$
& $\{f_2\}$
& $\{f_3\}$
\end{tabular}
\end{table}

\begin{proof}
Fix $s\in\{3,4,5,6\}$, let $W$ be an $s$-uniform witness,
and put $q=|W\cap F|$. By Lemma~\ref{sec3:lem-bounds},
$q\le2$. The disjoint decomposition of $E_0$ gives
\[
|W|
=q+\sum_{j=1}^{3}|W\cap A_j|
  +\sum_{a=4}^{5}|W\cap D_a|
\le q+c_s.
\]
Proposition~\ref{sec3:prop-witnesses} gives an $s$-uniform
witness $S_s$, so $\beta_s(P_0)\ge |S_s|$, where
$(|S_3|,|S_4|,|S_5|,|S_6|)=(44,42,41,37)$.

We now distinguish four cases according to $s$.

\smallskip
\noindent\textbf{Case 1.} $s=3$.

The bound $|W|\le q+c_3\le2+42=44$, together with
the witness $S_3$, gives $\beta_3(P_0)=44$.
Suppose that $|W|=44$. Equality requires $q=2$ and
$|W\cap G|=42$. Since $|G|=42$, we have $G\subseteq W$.
Thus, for each $j=1,2,3$, both $\xi_j$ and $\eta_j$ belong to $W$.
The dependent sets $\{f_j,\xi_j,\eta_j\}$ therefore
exclude $f_1,f_2,f_3$ from $W$. The two elements of
$W\cap F$ must be $f_4,f_5$. Hence
$W=G\cup\{f_4,f_5\}=S_3$, proving uniqueness.

\smallskip
\noindent\textbf{Case 2.} $s=4$.

If $q=0$, then $|W|\le c_4=41<42\le\beta_4(P_0)$,
so $W$ cannot be a maximum $4$-uniform witness.
If $q=2$, write $W\cap F=\{f_p,f_r\}$ with $p\ne r$.
Since $f_p,f_r$ span $H_0$, the equality
$\eta_j-\xi_j=f_j\in H_0$ implies that
$\{f_p,f_r,\xi_j,\eta_j\}$ is dependent for $j=1,2,3$.
Thus at least one of $\xi_j,\eta_j$ is absent from $W$.
For $j=2,3$, this improves the bounds in
Table~\ref{sec3:tab-bounds} to
$|W\cap A_2|\le9$ and $|W\cap A_3|\le11$.
Consequently, by Table~\ref{sec3:tab-bounds},
\[
|W|\le2+7+9+11+6+6=41.
\]
If $q=1$, then $|W|\le c_4+1=42$.
Since $S_4$ has size $42$, we obtain
$\beta_4(P_0)=42$, and every maximum witness has $q=1$.

Let $W$ be a maximum $4$-uniform witness.
We determine the unique element of $W\cap F$.

\smallskip
\noindent\emph{Subcase 2.1.}
$W\cap F=\{f_j\}$ for some $j\in\{2,3\}$.

Since $\{f_j,\xi_j,\eta_j\}$ is dependent,
at least one of $\xi_j,\eta_j$ is absent from $W$.
Thus $|W\cap A_j|\le |A_j|-1$.
As the bound for $A_j$ in
Table~\ref{sec3:tab-bounds} is $|A_j|$,
we obtain $|W|\le1+(c_4-1)=41<42$,
a contradiction.

\smallskip
\noindent\emph{Subcase 2.2.}
$W\cap F=\{f_a\}$ for some $a\in\{4,5\}$.

Since $\{f_a\}\cup D_{a,h}$ is a dependent
four-element set for each $h=1,2$, neither
$D_{a,1}$ nor $D_{a,2}$ is contained in $W$.
Each petal has three elements, so
$|W\cap D_a|\le4$.
Since the bound for $D_a$ in
Table~\ref{sec3:tab-bounds} is six,
we obtain $|W|\le1+(c_4-2)=40<42$,
again a contradiction.

\smallskip
Therefore every maximum $4$-uniform witness
satisfies $W\cap F=\{f_1\}$.

\smallskip
\noindent\textbf{Case 3.} $s=5$.

If $q=0$, then $|W|\le c_5=40<41\le\beta_5(P_0)$, so $W$ cannot be a maximum $5$-uniform witness.
If $q=2$, write $W\cap F=\{f_p,f_r\}$ with $p\ne r$.
Since $f_p,f_r$ span $H_0$ and
$\sum_{x\in D_{a,h}}x=f_a\in H_0$,
the five-element set $\{f_p,f_r\}\cup D_{a,h}$
is dependent for each $a\in\{4,5\}$ and $h\in\{1,2\}$.
Thus no $D$-petal is contained in $W$.
As each petal has three elements,
$|W\cap D_4|,|W\cap D_5|\le4$. By Table~\ref{sec3:tab-bounds},
\[
|W|\le2+7+9+12+4+4=38.
\]
If $q=1$, then $|W|\le c_5+1=41$.
Since $S_5$ has size $41$, we obtain
$\beta_5(P_0)=41$, and every maximum witness has $q=1$.

Let $W$ be a maximum $5$-uniform witness.
We determine the unique element of $W\cap F$.

\smallskip
\noindent\emph{Subcase 3.1.} $W\cap F=\{f_1\}$.

By Table~\ref{sec3:tab-bounds}, $|W\cap A_1|\le7$.
Suppose that equality holds.
If $\xi_1\in W$, the dependent four-element sets
$\{\xi_1\}\cup B_{1,h}$ force both three-element
petals $B_{1,1}$ and $B_{1,2}$ to be incomplete. Hence
$|W\cap A_1|\le2+2+2=6$, a contradiction.
Thus $\xi_1\notin W$.

Since $A_1\setminus\{\xi_1\}$ consists of $\eta_1$
and the two three-element petals $B_{1,1}$ and $B_{1,2}$, it has exactly
seven elements. Equality would therefore require
$\eta_1\in W$ and $B_{1,1},B_{1,2}\subseteq W$.
Together with $f_1\in W$, this places the dependent
five-element set $\{f_1,\eta_1\}\cup B_{1,h}$ in $W$,
again a contradiction.

Therefore $|W\cap A_1|\le6$.
Combining this with the other bounds in
Table~\ref{sec3:tab-bounds} gives
$|W|\le1+6+9+12+6+6=40<41$.

\smallskip
\noindent\emph{Subcase 3.2.} $W\cap F=\{f_3\}$.

Since $\{f_3,\xi_3,\eta_3\}$ is dependent,
at least one of $\xi_3,\eta_3$ is absent from $W$.
Thus $|W\cap A_3|\le11$, and
$|W|\le1+7+9+11+6+6=40<41$,
a contradiction.

\smallskip
\noindent\emph{Subcase 3.3.}
$W\cap F=\{f_a\}$ for some $a\in\{4,5\}$.

For each $h=1,2$, the four-element set
$\{f_a\}\cup D_{a,h}$ is dependent.
Since $f_a\in W$, neither $D_{a,1}$ nor
$D_{a,2}$ is contained in $W$.
Thus $|W\cap D_a|\le4$, instead of the
bound six in Table~\ref{sec3:tab-bounds}.
Consequently, $|W|\le1+(c_5-2)=39<41$,
a contradiction.

\smallskip
Therefore every maximum $5$-uniform witness
satisfies $W\cap F=\{f_2\}$.

\smallskip
\noindent\textbf{Case 4.} $s=6$.

If $q=0$, then $|W|\le c_6=36<37\le\beta_6(P_0)$.
If $q=2$, write $W\cap F=\{f_p,f_r\}$ with $p\ne r$.
Since the five-element set $\{f_p,f_r\}\cup D_{a,h}$
is dependent for $a=4,5$ and $h=1,2$, every $D$-petal is incomplete in $W$, giving
$|W\cap D_4|,|W\cap D_5|\le4$.
Consequently, by Table~\ref{sec3:tab-bounds},
\[
|W|\le2+6+9+11+4+4=36.
\]
If $q=1$, then $|W|\le c_6+1=37$.
Since $S_6$ has size $37$, we obtain
$\beta_6(P_0)=37$, and every maximum witness has $q=1$.

Let $W$ be a maximum $6$-uniform witness.
We determine the unique element of $W\cap F$.

\smallskip
\noindent\emph{Subcase 4.1.} $W\cap F=\{f_1\}$.

We consider three possibilities for $\xi_1$ and $\eta_1$.

If $\xi_1\in W$, the dependent set
$\{f_1,\xi_1,\eta_1\}$ implies $\eta_1\notin W$.
Moreover, the dependent sets
$\{\xi_1\}\cup B_{1,h}$ force both three-element
petals $B_{1,1}$ and $B_{1,2}$ to be incomplete. Hence
$|W\cap A_1|\le1+2+2=5$.

If $\xi_1\notin W$ but $\eta_1\in W$,
the dependent five-element sets
$\{f_1,\eta_1\}\cup B_{1,h}$ force both $B_{1,1}$ and $B_{1,2}$
to be incomplete. Thus
$|W\cap A_1|\le1+2+2=5$.

If neither $\xi_1$ nor $\eta_1$ belongs to $W$,
then $W\cap A_1\subseteq B_{1,1}\cup B_{1,2}$.
Since $B_{1,1}\cup B_{1,2}$ is a dependent
six-element set, $W$ contains at most five
of its elements. Hence $|W\cap A_1|\le5$.

In all three possibilities, $|W\cap A_1|\le5$.
Therefore $|W|\le1+5+9+11+5+5=36<37$,
a contradiction.

\smallskip
\noindent\emph{Subcase 4.2.} $W\cap F=\{f_2\}$.

By Table~\ref{sec3:tab-bounds}, $|W\cap A_2|\le9$.
Suppose that equality holds.
If $\xi_2\in W$, the dependent five-element sets
$\{\xi_2\}\cup B_{2,h}$ force both four-element
petals $B_{2,1}$ and $B_{2,2}$ to be incomplete. Thus
$|W\cap A_2|\le2+3+3=8$, a contradiction.
Hence $\xi_2\notin W$.

The set $A_2\setminus\{\xi_2\}$ consists of $\eta_2$
and the two four-element petals $B_{2,1}$ and $B_{2,2}$, so it has nine elements.
Equality would therefore require $\eta_2\in W$ and
$B_{2,1},B_{2,2}\subseteq W$.
Together with $f_2\in W$, this places the dependent
six-element set $\{f_2,\eta_2\}\cup B_{2,h}$ in $W$,
again a contradiction.

Therefore $|W\cap A_2|\le8$, giving
$|W|\le1+6+8+11+5+5=36<37$.

\smallskip
\noindent\emph{Subcase 4.3.}
$W\cap F=\{f_a\}$ for some $a\in\{4,5\}$.

For each $h=1,2$, the four-element set
$\{f_a\}\cup D_{a,h}$ is dependent.
Since $f_a\in W$, both $D_{a,1}$ and
$D_{a,2}$ must be incomplete.
Thus $|W\cap D_a|\le4$, instead of the
bound five in Table~\ref{sec3:tab-bounds}.
Consequently, $|W|\le1+(c_6-1)=36<37$,
a contradiction.

\smallskip
Therefore every maximum $6$-uniform witness
satisfies $W\cap F=\{f_3\}$.
This completes the proof.
\end{proof}

\begin{corollary}\label{sec3:cor-deletion}
For $e\in E(P_0)$, define
\[
 s(e)=
 \begin{cases}
  3,&e\in S_3,\\
  4,&e=f_1,\\
  5,&e=f_2,\\
  6,&e=f_3.
 \end{cases}
\]
Then
\begin{equation}\label{sec3:eq-deletion}
 \beta_{s(e)}(P_0\setminus e)=\beta_{s(e)}(P_0)-1.
\end{equation}
\end{corollary}

\begin{proof}
By Proposition~\ref{sec3:prop-maxima}, every maximum
$s(e)$-uniform witness contains $e$.
Fix $e$, and put $s=s(e)$, $b=\beta_s(P_0)$, and $Q=P_0\setminus e$.
Deletion preserves independence on all subsets of $E_0\setminus\{e\}$.
Hence any $s$-uniform witness $W'$ in $Q$ is also an $s$-uniform
witness in $P_0$ and has size at most $b$. Since every maximum
$s(e)$-uniform witness contains $e$, we have $\beta_s(Q)\le b-1.$

Conversely, choose a maximum $s$-uniform witness $W$ in $P_0$.
It contains $e$, so $|W\setminus\{e\}|=b-1$.
Every subset $Y\subseteq W\setminus\{e\}$ with $|Y|\le s$ is
independent in $P_0$. It is therefore
independent in $Q$ as well. Proposition~\ref{sec3:prop-maxima}
gives $b\ge37$. Since $s\le6$, we have $b-1\ge36\ge s$.
Thus $W\setminus\{e\}$ is an
$s$-uniform witness in $Q$, and $\beta_s(Q)\ge b-1$.
Combining the two inequalities proves \eqref{sec3:eq-deletion}.
\end{proof}

\section{Proof of Theorem~\ref{sec1:main}}
\label{sec4:counterexample}

\begin{proof}[Proof of Theorem~\ref{sec1:main}]
By Proposition~\ref{sec3:prop-basic}, $P_0$ is a simple
matroid of rank six, and $F$ is a rank-two flat.
Proposition~\ref{sec3:prop-witnesses} gives a
$6$-uniform witness $S_6$ with $|S_6|=37>3(6-1)$.
Thus Lemma~\ref{sec2:lem-large-witness} implies that
$\mathcal T$ is a tangle of order seven with
$(P_0)_{\mathcal T}=P_0$, and for $4\le i\le7$, 
\[
\operatorname{breadth}_{P_0}(\mathcal T_i)
=\beta_{i-1}(P_0).
\]

Suppose that a minor $N$ of $P_0$ admits an order-$i$ tangle $\mathcal{S}_i$ for each $i \in \{4,5,6,7\}$ such that
\[
\operatorname{breadth}_{N}(\mathcal{S}_i) = \operatorname{breadth}_{P_0}(\mathcal{T}_i).
\]
Assume, for a contradiction, that $N$ is a proper minor of $P_0$. Choose $e \in E(P_0) \setminus E(N)$ and put $s = s(e)$, as defined in Corollary 3.7. Then $s+1 \in \{4,5,6,7\}$, and $P_0|E(N)$ is a restriction of $P_0 \setminus e$. By Corollary 2.4,
\[
\begin{aligned}
\beta_s(P_0) &= \operatorname{breadth}_{P_0}(\mathcal{T}_{s+1}) = \operatorname{breadth}_{N}(\mathcal{S}_{s+1}) \\
&\le \beta_s(P_0|E(N)) \le \beta_s(P_0 \setminus e) \\
&= \beta_s(P_0) - 1,
\end{aligned}
\]
where the first equality follows from Lemma 2.5 and the last from Corollary 3.7. This contradiction shows that $N = P_0$.

Finally, we show $P_0$ is not
weakly $4$-connected. Since $S_6$ is
a $6$-uniform witness, $S_6\setminus\{f_3\}\subseteq G$ contains
an independent six-element set.
As $r(P_0)=6$, it follows that $r_{P_0}(G)=6$.
Since $r_{P_0}(F)=2$ and $E(P_0)=F\sqcup G$,
\[
\lambda_{P_0}(F)=2+6-6=2.
\]
But $|F|=5>4$ and $|G|=42>4$, so $P_0$ is not
weakly $4$-connected.

Therefore Conjecture~\ref{sec1:conjecture} is false.
\end{proof}

\section{Internally \texorpdfstring{$4$}{4}-connected minors}
\label{sec5:internal}

We prove Theorem~\ref{sec1:main-internal} using the following
construction.

\begin{construction}\label{sec5:construction}
Let $P_1$ be the vector matroid represented over $\mathbb Q$ by
\begin{equation}\label{sec5:eq-matrix}
\bordermatrix{
& {e_{12}}
& {e_{13}}
& {e_{14}}
& {e_{23}}
& {e_{24}}
& {e_{34}}
& {p_1}
& {p_2}
& {p_3}
& {p_4}
\cr
& 0&0&0&1&1&1&0&1&1&1\cr
& 0&1&1&0&0&-1&1&0&3&2\cr
& 1&0&-1&0&-1&0&1&2&0&-3\cr
}
\end{equation}
Let $E_1$ be the set of column labels, so that
$E(P_1)=E_1$.
Put $Z=\{p_1,p_2,p_3,p_4\}$ and define
\begin{equation*}\label{sec5:eq-lines}
\begin{aligned}
L_1&=\{p_1,e_{12},e_{13},e_{14}\},&
L_2&=\{p_2,e_{12},e_{23},e_{24}\},\\
L_3&=\{p_3,e_{13},e_{23},e_{34}\},&
L_4&=\{p_4,e_{14},e_{24},e_{34}\}.
\end{aligned}
\end{equation*}
\end{construction}

\begin{proof}[Proof of Theorem~\ref{sec1:main-internal}]
Every column in~\eqref{sec5:eq-matrix} is nonzero, and
any two distinct columns are linearly independent.
The columns labelled $e_{12},e_{13},e_{23}$ are linearly
independent, so $P_1$ is simple of rank three.

A direct calculation shows that $L_1,L_2,L_3,L_4$
are lines, and all other lines have two elements. Moreover,
$L_i\cap L_j=\{e_{ij}\}$ for $1\le i<j\le4$. Hence any three lines
cover at most nine elements of $E_1$. By~\cite[Lemma~76]{BJO25},
\[
\mathcal T=\{A\subseteq E_1:r_{P_1}(A)<3\}
\]
is the unique tangle of order four in $P_1$, and
$(P_1)_{\mathcal T}=P_1$.

To compute its breadth, let $W$ be a $3$-uniform witness in
$P_1$. Since each $L_i$ has rank two, $|W\cap L_i|\le2$.
Each $e_{ij}$ belongs to two of the four lines, and each $p_i$
belongs to one. Hence
\[
8\ge\sum_{i=1}^{4}|W\cap L_i|=2|W\setminus Z|+|W\cap Z|=2|W|-|W\cap Z|.
\]
As $|W\cap Z|\le4$, we obtain $|W|\le6$.
Let $U=Z\cup\{e_{12},e_{34}\}$.
Then $|U|=6$ and $|U\cap L_i|=2$ for each
$i\in\{1,2,3,4\}$. Since every dependent three-element subset of $E_1$
is contained in one of $L_1,L_2,L_3,L_4$, every three-element subset of $U$ is independent.
Thus $U$ is a $3$-uniform witness, and $\operatorname{breadth}_{P_1}(\mathcal T)=\beta_3(P_1)=6$.

It remains to show that no internally $4$-connected minor of
$P_1$ admits a tangle of order four.
First, $P_1$ itself is not internally $4$-connected.
We have $r_{P_1}(L_1)=2$, $r_{P_1}(E_1\setminus L_1)=3$ and
$\lambda_{P_1}(L_1)=2$.
Since $|L_1|=4$ and $|E_1\setminus L_1|=6$, thus
$P_1$ is not internally $4$-connected.

Now let $N=(P_1/K)\setminus D$ be a minor with a tangle
$\mathcal S$ of order four, where $K\cap D=\varnothing$.
Lemma~\ref{sec2:lem-rank-comparison} gives $r(N)\ge3$.
If $K\ne\varnothing$, simplicity of $P_1$ gives
$r_{P_1}(K)\ge1$, and therefore
$r(N)\le r(P_1/K)=3-r_{P_1}(K)\le2$, a contradiction.
Thus $K=\varnothing$ and $N=P_1\setminus D$.

Suppose that $D\ne\varnothing$, and choose $x\in D$.
If $x=p_i$, the three lines $L_j$ with $j\ne i$
cover $E_1\setminus\{x\}$.
If $x=e_{ij}$, let
$\{k,l\}=\{1,2,3,4\}\setminus\{i,j\}$. Then $E_1\setminus\{x\}=L_k\cup L_l\cup\{p_i,p_j\}.$
In either case, there are sets $H_1,H_2,H_3\subseteq E_1$,
each of rank at most two in $P_1$, such that
\[
E_1\setminus\{x\}=H_1\cup H_2\cup H_3.
\]
For $h=1,2,3$, put $X_h=H_h\cap E(N)$.
Since $x\in D$ and $N=P_1\setminus D$, we have $E(N)=X_1\cup X_2\cup X_3,$ and $r_N(X_h)\le r_{P_1}(H_h)\le2.$
By Lemma~\ref{sec2:lem-rank-comparison},
$X_1,X_2,X_3\in\mathcal S$, contrary to axiom~(T3).
Thus $D=\varnothing$ and $N=P_1$.

Since $P_1$ is not internally $4$-connected, no internally
$4$-connected minor of $P_1$ admits a tangle of order four.
Consequently, Conjecture~\ref{sec1:conjecture-internal} is false.
\end{proof}


\begin{thebibliography}{10}

\bibitem{BJO25}
N.~Brettell, S.~Jowett, J.~Oxley, C.~Semple, and G.~Whittle,
\emph{What is a $4$-connected matroid?},
Electron. J. Combin. \textbf{32} (2025), no.~2, Paper~P2.1.
\doi{10.37236/12467}.

\bibitem{CK25}
J.~Carmesin and J.~Kurkofka,
\emph{Characterising $4$-tangles through a connectivity property},
Electron. J. Combin. \textbf{32} (2025), no.~3, Paper~P3.26.
\doi{10.37236/12367}.

\bibitem{Dha96}
J.~S.~Dharmatilake,
\emph{A min-max theorem using matroid separations},
in \emph{Matroid Theory} (Seattle, WA, 1995),
J.~E.~Bonin, J.~G.~Oxley, and B.~Servatius (eds.),
Contemp. Math., vol.~197, American Mathematical Society,
Providence, RI, 1996, pp.~333--342.

\bibitem{GGRW06}
J.~Geelen, B.~Gerards, N.~Robertson, and G.~Whittle,
\emph{Obstructions to branch-decomposition of matroids},
J. Combin. Theory Ser. B \textbf{96} (2006), no.~4, 560--570.
\doi{10.1016/j.jctb.2005.11.001}.

\bibitem{GGW09}
J.~Geelen, B.~Gerards, and G.~Whittle,
\emph{Tangles, tree-decompositions and grids in matroids},
J. Combin. Theory Ser. B \textbf{99} (2009), no.~4, 657--667.
\doi{10.1016/j.jctb.2007.10.008}.

\bibitem{GvZ15}
J.~Geelen and S.~H.~M.~van Zwam,
\emph{Matroid $3$-connectivity and branch width},
J. Combin. Theory Ser. B \textbf{112} (2015), 104--123.
\doi{10.1016/j.jctb.2014.12.003}.

\bibitem{Hall15}
D.~Hall,
\emph{A characterization of tangle matroids},
Ann. Comb. \textbf{19} (2015), no.~1, 125--130.
\doi{10.1007/s00026-015-0259-8}.

\bibitem{Oxley11}
J.~Oxley,
\emph{Matroid Theory}, 2nd ed.,
Oxford Graduate Texts in Mathematics, vol.~21,
Oxford University Press, Oxford, 2011.
\doi{10.1093/acprof:oso/9780198566946.001.0001}.

\bibitem{RS91}
N.~Robertson and P.~D.~Seymour,
\emph{Graph minors. X. Obstructions to tree-decomposition},
J. Combin. Theory Ser. B \textbf{52} (1991), no.~2, 153--190.
\doi{10.1016/0095-8956(91)90061-N}.

\bibitem{Tutte66}
W.~T.~Tutte,
\emph{Connectivity in matroids},
Canad. J. Math. \textbf{18} (1966), 1301--1324.
\doi{10.4153/CJM-1966-129-2}.

\end{thebibliography}
\end{document}